\documentclass[12pt,reqno]{amsart} 
\usepackage{amsmath,tikz, tikz-cd}
\usepackage[utf8]{inputenc}
\usepackage[colorlinks=true,linkcolor=blue,citecolor=blue]{hyperref}
\usepackage[normalem]{ulem}
\usepackage{amsmath}
\usepackage{amssymb}
\usepackage{amsthm}
\usepackage{mathtools}
\usepackage{verbatim}
\usepackage{xcolor}
\usepackage{fullpage}
\usepackage[skip=3pt, indent=20pt]{parskip}
\usepackage{caption}
\usepackage{comment}
\usepackage{marginnote}
\usepackage{cite}
\numberwithin{equation}{section}
\theoremstyle{plain}
\newtheorem{theorem}{Theorem}[section]
\newtheorem{lemma}[theorem]{Lemma}
\newtheorem{proposition}[theorem]{Proposition}
\newtheorem{corollary}[theorem]{Corollary}

\theoremstyle{definition}
\newtheorem{definition}[theorem]{Definition}
\newtheorem{example}[theorem]{Example}

\theoremstyle{remark}
\newtheorem{remark}[theorem]{Remark}
\newcommand{\Z}{\mathbb{Z}}

\newcommand{\I}{\mathbb{I}}
\newcommand{\Pin}{\operatorname{Pin}}

\newcommand{\Peak}{\operatorname{Peak}}

\newcommand{\p}{\mathfrak{p}}

\newcommand\APS{\mathsf{APS}}

\newcommand{\ceil}[1]{\left\lceil#1\right\rceil}
\newcommand{\floor}[1]{\left\lfloor#1\right\rfloor}

\newcommand{\newword}[1]{\emph{\textbf{#1}}}

\title{Pinnacles for Complex Reflection Groups}

\author{Aaron Burnham-Schmidt}
\address[A. Burnham-Schmidt]{Department of Mathematics, University of California, Berkeley, CA,  USA}
\email{\textcolor{blue}{\href{mailto:afburnham@berkeley.edu}{afburnham@berkeley.edu}}}
\thanks{}

\author{Nicolle Gonz\'{a}lez}
\address[N. Gonz\'{a}lez] {Department of Mathematics,
University of British Columbia, Vancouver, BC, Canada }
\address{Institute for Computational and Experimental Research in Mathematics, Providence, RI, USA}
\email{\textcolor{blue}{\href{mailto:nicolle\_sandoval\_gonzalez@brown.edu}{nicolle\_sandoval\_gonzalez@brown.edu}}}
\thanks{The authors thank John Lentfer for helpful conversations}

\begin{document}
\begin{abstract}
We study, characterize, and enumerate the admissible pinnacle sets of nonexceptional complex reflection groups $G(m,p,n)$, which include all generalized symmetric groups $\Z_m \wr S_n$ as special cases. 
This generalizes the work of Davis--Nelson--Petersen--Tenner for symmetric groups $S_n$ and Gonz\'alez--Harris--Rojas Kirby--Smit Vega Garcia--Tenner for signed symmetric groups $\Z_2 \wr S_n$.
As a consequence, we prove a conjecture of Gonz\'alez--Harris--Rojas Kirby--Smit Vega Garcia--Tenner for pinnacles of signed permutations.
\end{abstract}

\maketitle

%\tableofcontents
\section{Introduction}
\emph{Complex reflection groups} are an important family of finite groups that contain all Coxeter groups and symmetry groups of regular polyhedra as special cases. These groups arise as automorphism groups of complex vector spaces and have been central to many modern developments in algebra and combinatorics. In their groundbreaking 1954 paper \cite{Shephard_Todd_1954}, Sheppard and Todd classified all irreducible complex reflection groups into either one of 34 exceptional cases or a member of the infinite family of groups $G(m,p,n)$ of order $m^n n!/p$, which depend on positive integers $m,p,n$ with $p$ dividing $m$. 
These nonexceptional complex reflection groups $G(m,p,n)$ are precisely the reflection groups for the complex polytopes considered in \cite[pg. 378]{Shep53}.
When $m=p=1$ the group $G(1,1,n)$ is the classical type $A_n$ symmetric group $S_n$. More generally, we have that:
\begin{itemize}
\item $G(m,1,n)$ is the \emph{generalized symmetric group} $\Z_m \wr S_n$.
\item $G(2,1,n)$ is the type $B_n$ \emph{signed symmetric group}.
\item $G(2,2,n)$ is the type $D_n$ \emph{signed symmetric group}.
\item $G(m,m,2)$ is the \emph{dihedral group} of order $2m$.
\item $G(2,2,2)$ is the \emph{Klein 4 group}.
\end{itemize}

In this paper we study and enumerate the admissible pinnacle sets of $G(m,n,p)$. The notion of a \emph{pinnacle set} was first studied by Davis--Nelson--Petersen--Tenner \cite{PinnaclesTypeA} in the context of $S_n$ and motivated by the work of Billey--Burdzy--Sagan \cite{BBS} on so-called \emph{peak sets}. Given a permutation $w = w(n)\dots w(1)$ the peak set of $w$ is given by
\[
\Peak(w):= \{ i \in \{1,\dots,n\} \; |\; w(i+1)<w(i)> w(i-1) \}.
\]
Similarly, the \newword{pinnacle set} of $w$ is the set of integers:
\[
\Pin(w):= \{ w(i) \in \{1,\dots,n\} \; |\; i \in \Peak(w) \}.
\]
Informally, if we view a permutation in $S_n$ as a function $w:\{1,\dots,n\} \to \{1,\dots,n\}$, then $\Peak(w)$ is the set of $x$-values and $\Pin(w)$ is the set of $y$-values of the local maxima of the function $w$.

When studying pinnacle sets, several questions naturally arise. The first is determining whether a given subset $P \subseteq \{1,\dots,n\}$ is \newword{admissible} in $S_n$; that is, determining whether there exists a permutation $w \in S_n$ with $\Pin(w)=P$ and if so giving a characterization of when this occurs. The second problem is, for a given $n$, enumerating the number of admissible pinnacle sets of a given cardinality.  

In \cite{PinnaclesTypeA} Davis--Nelson--Petersen--Tenner gave a complete solution to the questions above by providing a closed formula for the number of admissible pinnacle sets with a given maximal value and a recursive formula for the number of permutations with a given pinnacle set. Their work led to a flurry of results searching for more computationally efficient formulas \cite{DHHIN,FNT, DLMSSS, fang, Minnich} and other related properties \cite{rusu, rusu-tenner}. 

Pinnacle sets have also been extended to other finite groups. In \cite{PinSignedPermutations} the second author alongside Harris, Rojas Kirby, Smit Vega Garcia, and Tenner extended the notion of pinnacle sets to type $B_n$ and $D_n$ signed symmetric groups where they characterized and enumerated pinnacle sets for signed permutations. In \cite{MesasStirling} the same authors introduced the related notion of \emph{mesa sets} for Stirling permutations and gave similar enumeration results in terms of Catalan numbers. 
Independently, Peak sets have also been studied and generalized to various contexts \cite{CDOPZ,DHIP,DHIOS,DEHIMO}. 

In this article we give a complete characterization of admissible pinnacle sets and various recursive and closed enumeration formula for the number of admissible pinnacle sets of any given cardinality $d$ for all generalized symmetric groups and all complex reflection groups of the form $G(m,p,n)$. 

\subsection*{Pinnacles for Generalized Symmetric Groups} For any positive integers $m,n$, the \newword{generalized symmetric group} arises as the wreath product group $\Z_m \wr S_n$. Let $\APS_d(m,n)$ denote the collection of admissible pinnacle sets in $\Z_m \wr S_n$ of cardinality less than or equal to $d$. It is easy to see that any admissible set must have cardinality $d \leq \floor{(n-1)/2}$ (see Lemma \ref{lem:maxLength}). In order to enumerate these sets, we utilize certain natural embeddings of these sets and in Theorem \ref{thm: pinCharacter} and Corollary \ref{cor:pinCharacter} give two distinct complete characterizations of when a given subset is admissible in $\Z_m \wr S_n$. With this in hand, we then count $\#\APS_d(m,n)$ in four distinct ways: two recursive and two via closed formulas. 

The first recursion, found in Theorem \ref{thm:recursion}, vastly generalizes a similar recursion in \cite[Proposition 3.11]{PinSignedPermutations} and states that for any positive integers $m,n$ and $d\leq \floor{\frac{n-1}{2}}$:
\[
\#\APS_d(m,n) =\sum_{i=0}^d\binom{n}{i}\#\APS_{d-i}(m-1,n-i).
\]
This recursion leads us to the closed formula in Theorem \ref{thm:pinn-formula} expressing $\#\APS_d(m,n)$ as a partial binomial sum,
\[
\#\APS_d(m,n)= \sum_{i=0}^{d} \binom{n}{i}m^i(-1)^{i+d},
\]
or alternatively as a strictly positive sum in Corollary \ref{cor:pinformula} as,
\[
\#\APS_d(m,n)= \sum_{k=0}^d(m-1)^k \binom{n}{k}\binom{n-k-1}{d-k}.
    \]
 In \cite[Conjecture 5.1]{PinSignedPermutations} Gonz\'alez--Harris--Rojas Kirby--Smit Vega Garcia--Tenner proposed that for the signed symmetric group $\Z_2 \wr S_n$ the number of admissible pinnacle sets recovered the OEIS sequence A119258. In Theorem \ref{thm:recursion2} we prove the recursion below holds for all generalized symmetric groups 
 \[
 \#\APS_d(m,n) = m\;\#\APS_{d-1}(m,n-1)+\#\APS_d(m,n-1),
 \]
 thus verifying \cite[Conjecture 5.1]{PinSignedPermutations} as a special case when $m=2$. 

\subsection*{Pinnacles for Complex Reflection Groups} In order to extend these results to the nonexceptional complex reflection groups we utilize the fact that any $G(m,p,n)$ can be realized as a particular normal subgroup of $\Z_m \wr S_n = G(m,1,n)$. Then, appealing to similar techniques as those in \cite{PinSignedPermutations}, in Theorem \ref{thm:APS-complex-equal} we prove that the number of admissible pinnacle sets in $G(m,p,n)$ of cardinality less than or equal to $d$, denoted $\APS_d(m,p,n)$, coincides with those in $\Z_m \wr S_n $. Namely, for any $p|m$ and $d<\ceil{(n-1)/2}$ we show that:
\[
\APS_d(m,p,n) = \APS_d(m,n).
\]
Thus, any of the enumeration results for generalized symmetric groups given above also enumerate the number of admissible pinnacle sets for almost all (nonexceptional) complex reflection groups. 

The remaining case, when $d=\ceil{(n-1)/2}=\floor{(n-1)/2}$ is maximal, is only achieved when $n$ is odd: namely, when $n=2r+1$ and $d=r$. Unfortunately, this situation has proven to be much more resistant to our methods so we have been unable to provide a closed formula for $\#\APS_r(m,p,2r+1)$. Nonetheless, although we do not compute it explicitly, in Theorem \ref{thm:odd-reduction} and Corollary \ref{cor:reduction} we prove that 
\[
\#\APS_r(m,p,2r+1) = \#\APS_r(p,p,2r+1) + \sum_{i=0}^r\binom{2r+1}{i}(m^i-p^i)(-1)^{i+r}.
\]
Thus, we successfully reduce the computation of $\APS_r(m,p,2r+1)$ to the case when $m=p$.

\subsection*{Structure of the Paper} This article is organized as follows. In \S \ref{Sec:GeneralizedSymGroup} we introduce the necessary background and notation and define pinnacles and admissible pinnacle sets for elements in $\Z_m \wr S_n$. In \S\ref{subsec:characterization} we give the two characterizations of $\APS_d(m,n)$ via Theorem \ref{thm: pinCharacter} and Corollary \ref{cor:pinCharacter}. It is in \S\ref{subsec:enumeration} that we give the main enumeration results in Theorems \ref{thm:recursion}, \ref{thm:pinn-formula}, \ref{thm:recursion2}, and Corollary \ref{cor:pinformula}. Finally, in \S\ref{sec:ComplexReflGroup} we use the results from \S\ref{Sec:GeneralizedSymGroup} to enumerate the admissible pinnacle sets for complex reflection groups in Theorems \ref{thm:APS-complex-equal} and \ref{thm:odd-reduction} and Corollaries \ref{cor:APS-even} and \ref{cor:reduction}.

\section{Pinnacles for Generalized Symmetric Groups} \label{Sec:GeneralizedSymGroup}

For positive integers $m,n$, the \emph{generalized symmetric groups} are the wreath products $\Z_m \wr S_n$ with underlying sets $\Z_m^n \times S_n$ and group operation given by \label{def: groupop}
\[ (a_1,\dots_,a_n, \sigma) \cdot (b_1,\dots,b_n, \omega) = (a_1+b_{\sigma(1)}, \dots, a_n+b_{\sigma(n)}, \sigma\omega).\]
It is known that $\Z_m \wr S_n$ is isomorphic to the group of $n\times n$ permutation matrices whose entries are $m^{th}$ roots of unity $\{\xi^i\;| \xi\neq 1,\; 0\leq i \leq m-1\}$ where $\xi= e^{2\pi i/m}$. 
For $n \in \Z_{>0}$, denote by $[n]:=\{1,\dots n\}$ the set of positive integers up to $n$, and $\xi^i[n]:=\{ \xi^i(1), \xi^i(2),\dots, \xi^i(n)\}$ the analogous set scaled by $\xi^i$. Thus, $\Z_m \wr S_n$ can be equivalently be defined as the set of bijections $w$ on the set $\I_n^m:=\bigcup_{i=0}^{m-1} \xi^i[n]$ satisfying $w(\xi^i x) = \xi^{i} w(x)$ for all $x \in [n]$ and $0\leq i\leq m-1$, where we impose the total ordering:
\[
\xi^{m-1}(n) \prec \xi^{m-1}(n-1)\prec \dots \prec\xi^{i+1}(1) \prec \xi^{i}(n)\prec \dots \prec \xi^{1}(1)\prec n \prec \dots \prec 2\prec 1.
\]
In particular, the condition that $w(\xi^i x) = \xi^{i} w(x)$ for all $x$ and $i$ implies that any permutation $w \in \Z_m \wr S_n$ is uniquely determined by its action on $[n]$. 
Hence, we without loss of generality we identify $w$ with its image, and use one line notation to write $w=w(n)w(n-1)\dots w(1)$ to denote the image of $w$.

\begin{example}
  Let $m=3$ and $n=6$. Then the word 
  \[\xi^2(6)\xi(1)\xi(3)\xi^2(4)\xi^0(2)\xi^2(5) \]
  denotes the permutation $w \in \Z_3 \wr S_6$ sending 
  \[
   6 \mapsto \xi^2(6)\; ,\; 5 \mapsto \xi(1) \;,\; 4 \mapsto \xi(3) , 3 \mapsto \xi^2(4) \;,\; 2 \mapsto \xi^0(2)\;,\; 1 \mapsto\xi^2(5).
  \]
\end{example}

\begin{definition} For a complex number $z \in \mathbb{C}$ denote by $\Vert z \Vert$ the Euclidean norm of $z$. Then, for any subset $P \subseteq \I_n^m$, write $\#P$ for its cardinality and define the set:
\[
|P|:=\{ \Vert x \Vert\;|\; x \in P\}= \{ x \in [n] \;| \; \xi^i(x) \in P \text{ for some } i\}.
\]
\end{definition}

So then, let $w=w(n)\dots w(1)$ be any generalized permutation in  $\Z_m \wr S_n$. 
\begin{definition}
A generalized permutation $w \in\Z_m \wr S_n$ has a \newword{pinnacle} at $w(i)$ if $w(i+1)\prec w(i) \succ w(i-1)$. The \newword{Pinnacle set} of $w$ is the collection
\[
\Pin(w):= \{ w(i) \in \I_n^m \; |\; w(i+1)\prec w(i)\succ w(i-1) \}.
\]
\end{definition}

\begin{remark}
 We will often write the one line presentation of a generalized permutation with the pinnacle values slight raised above the rest, as shown in Example \ref{ex:witness}, to highlight the topography of the permutation. 
\end{remark}

\begin{definition} A subset $P \subseteq \I_n^m$ is an \newword{admissible pinnacle set} for $\Z_m\wr S_n$ if  there exists a generalized permutation $w \in\Z_m\wr S_n$, termed a \newword{witness} for $P$, such that $\Pin(w) = P$.
\end{definition}

In particular, this means that a subset $P\subset \I_n^m$ is \emph{not} an admissible pinnacle set precisely when there exist no generalized permutations $w \in \Z_m\wr S_n$ that are witness permutations for $P$. 

\begin{definition}\label{def:APS}
For any integer $d\geq 0$, let $\APS_d(m,n)$ denote the collection of admissible pinnacle sets $P \subseteq \I_n^m$ for $ \Z_m \wr S_n$ such that $\#P \leq d$, and set $\APS(m,n) := \bigcup_d \APS_d(m,n)$.
\end{definition}

In particular, Definition \ref{def:APS} implies that $\APS(m,n)$ admits the following filtration:
\begin{equation}\label{eq:APSfiltration}
\APS_0(m,n) \subset \APS_1(m,n) \subset\dots \subset \APS_d(m,n)\subset \APS_{d+1}(m,n)\subset \dots \subset \APS(m,n).
\end{equation}

\begin{figure}[ht]
\begin{tikzpicture}[yscale=.5, xscale=1.2]
\draw[black!30, dotted, thick] (-.5,-10.5) -- (10.5,-10.5);
\draw[black!30, dotted, thick] (-.5,-0.5) -- (10.5,-0.5);
\draw[thick] (0.7,10)--(0.7,-20.5)--(10.5,-20.5);
\node at (1,-21)[scale=.75]{$10$};
\node at (2,-21)[scale=.75]{$9$};
\node at (3,-21)[scale=.75]{$8$};
\node at (4,-21)[scale=.75]{$7$};
\node at (5,-21)[scale=.75]{$6$};
\node at (6,-21)[scale=.75]{$5$};
\node at (7,-21)[scale=.75]{$4$};
\node at (8,-21)[scale=.75]{$3$};
\node at (9,-21)[scale=.75]{$2$};
\node at (10,-21)[scale=.75]{$1$};
\node at (0,0)[scale=.75]{$10$};
\node at (0,1)[scale=.75]{$9$};
\node at (0,2)[scale=.75]{$8$};
\node at (0,3)[scale=.75]{$7$};
\node at (0,4)[scale=.75]{$6$};
\node at (0,5)[scale=.75]{$5$};
\node at (0,6)[scale=.75]{$4$};
\node at (0,7)[scale=.75]{$3$};
\node at (0,8)[scale=.75]{$2$};
\node at (0,9)[scale=.75]{$1$};
\begin{scope}[shift={(0,-10)}]
\node at (0,0)[scale=.75]{$\xi(10)$};
\node at (0,1)[scale=.75]{$\xi(9)$};
\node at (0,2)[scale=.75]{$\xi(8)$};
\node at (0,3)[scale=.75]{$\xi(7)$};
\node at (0,4)[scale=.75]{$\xi(6)$};
\node at (0,5)[scale=.75]{$\xi(5)$};
\node at (0,6)[scale=.75]{$\xi(4)$};
\node at (0,7)[scale=.75]{$\xi(3)$};
\node at (0,8)[scale=.75]{$\xi(2)$};
\node at (0,9)[scale=.75]{$\xi(1)$};
\end{scope}
\begin{scope}[shift={(0,-20)}]
\node at (0,0)[scale=.75]{$\xi^2(10)$};
\node at (0,1)[scale=.75]{$\xi^2(9)$};
\node at (0,2)[scale=.75]{$\xi^2(8)$};
\node at (0,3)[scale=.75]{$\xi^2(7)$};
\node at (0,4)[scale=.75]{$\xi^2(6)$};
\node at (0,5)[scale=.75]{$\xi^2(5)$};
\node at (0,6)[scale=.75]{$\xi^2(4)$};
\node at (0,7)[scale=.75]{$\xi^2(3)$};
\node at (0,8)[scale=.75]{$\xi^2(2)$};
\node at (0,9)[scale=.75]{$\xi^2(1)$};
\end{scope}
\node (a) at (1,9){$\bullet$};
\node (b) at (2,6){$\bullet$};
\node (c) at (3,8){$\bullet$};
\node (d) at (4,-7){$\bullet$};
\node (e) at (5,-19){$\bullet$};
\node (f) at (6,-3){$\bullet$};
\node (g) at (7,-20){$\bullet$};
\node (h) at (8,-8){$\bullet$};
\node (i) at (9,4){$\bullet$};
\node (j) at (10,-6){$\bullet$};
\draw[thick] (a)--(b)--(c)--(d)--(e)--(f)--(g)--(h)--(i)--(j);
% \node at (8,8){$\bullet$};
%
\draw[red, thick] (c) circle (10pt);
\draw[red, thick] (f) circle (10pt);
\draw[red, thick] (i) circle (10pt);
%\draw[blue] (4,-10) circle (8pt);
%\draw[blue] (6,-11) circle (8pt);
\end{tikzpicture} 
\caption{The graph for the witness of $P =\{\xi^0(2),\xi^1(3),\xi^0(5)\}$ given by
$w =\xi^0(1)\xi^0(4)\xi^0(2)\xi^1(7)\xi^2(9)\xi^1(3)\xi^2(10)\xi^1(8)\xi^0(5)\xi^1(6) \in \Z_3 \wr S_{10}$. The pinnacle values are circles in red.
}\label{fig:witness}
\end{figure}
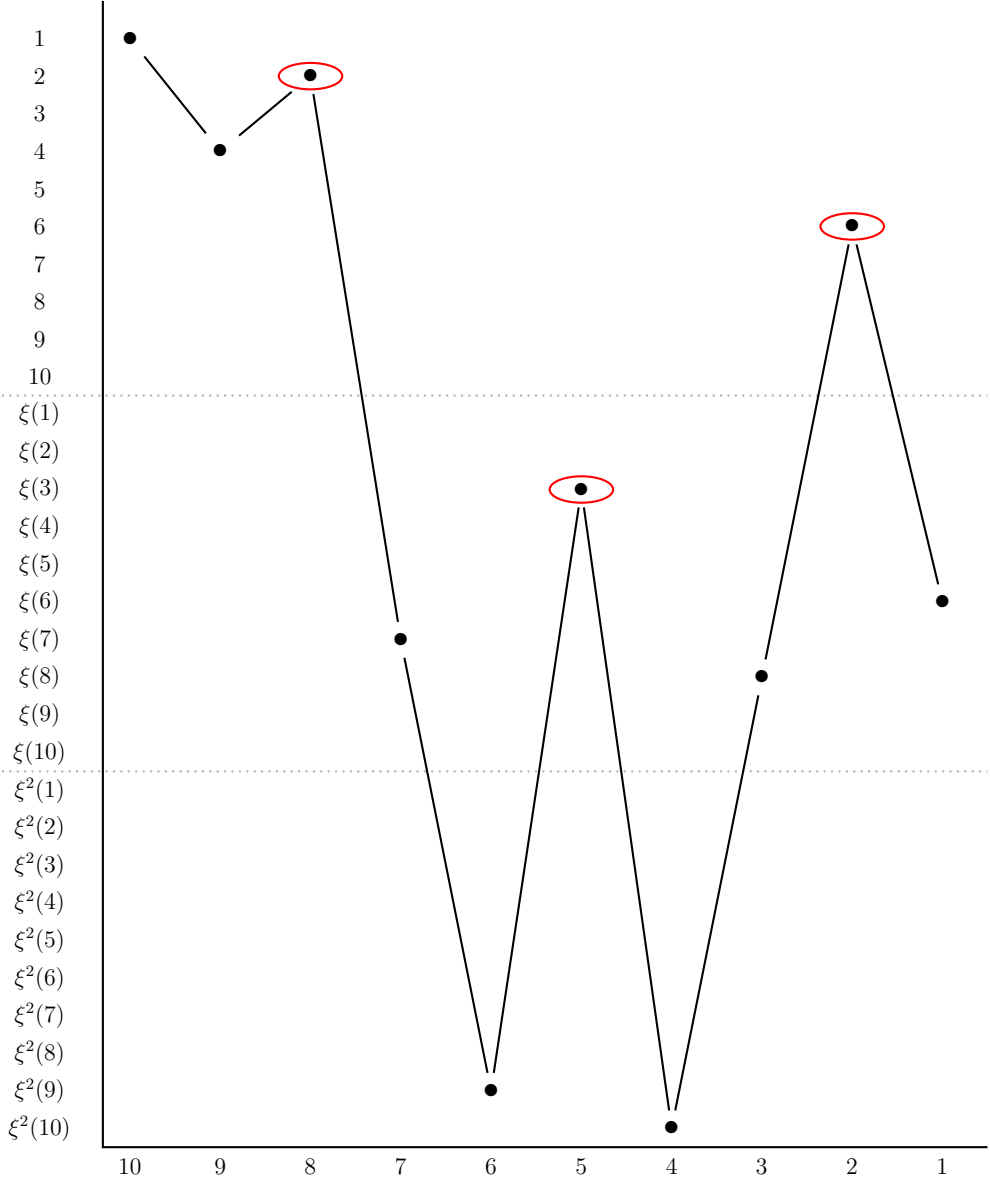

\begin{example}\label{ex:noncanonical witness}
    Consider $P =\{\xi^0(2),\xi^1(3),\xi^0(5)\} \subset \I_{10}^3$. Then $P$ has a witness in $\Z_3 \wr S_{10}$ (see Figure \ref{fig:witness}) given by
\[w =\xi^0(1)\xi^0(4)\xi^0(2)\xi^1(7)\xi^2(9)\xi^1(3)\xi^2(10)\xi^1(8)\xi^0(5)\xi^1(6).\]
Thus, $P \in \APS_3(3,10)$.
\end{example}

\begin{example}\label{ex:witness}
Consider $P = \{\xi^4(3)\prec \xi^3(5)\prec \xi^2(2)\}$. Then the permutation 
\[w=\;\begin{matrix}
   &\xi^2(2)&&\xi^4(3)&&\xi^3(5)&
\\
\xi^4(7)&&\xi^4(6)&&\xi^4(4)&&\xi^4(1)& \xi^4(8)& \xi^4(9)& \xi^4(10)
\end{matrix}\]
is a witness permutation for $P$ since $\Pin(w)=P$. Thus, $P \in \APS_3(4,10)$.
In particular, witnesses are not unique as the permutation $\sigma$ below also satisfies $\Pin(w)=P$, 
\[\sigma=\;\begin{matrix}
    &&&\xi^3(5)&&\xi^2(2)&& \xi^4(3) \\ 
    \xi^3(9) & \xi^4(8) & \xi^3(10)&&\xi^3(6)&&\xi^4(4)&&\xi^4(1)& \xi^1(7).
\end{matrix}
\]
\end{example}

\begin{example}\label{ex:trivial}
For any choice of $m,n$ the empty set is always an admissible pinnacle set for $\Z_m\wr S_n$ since we always have that $\Pin(id)=\emptyset$. 
\end{example}

\begin{example}\label{ex:inadmissible}
    Consider the set $S = \{\xi^4(3) \prec \xi^2(3) \prec \xi^0(1)\} \subset \I_7^5$. We claim $S$ is not an admissible pinnacle set in $\Z_5 \wr S_7$. To see this, suppose $S$ was admissible with witness $w$. Then necessarily, $w(i) = \xi^4(3)$ and $w(j) = \xi^2(3)$ for some $i\neq j \in \{ 1, \dots, 7\}$. Since $w(\xi^k x) = \xi^k(x)$ for any $k$ and $x$, then $w(\xi^1 i) = w(\xi^3 j) = 3$, contradicting the fact that $w$ is a bijection on $\I_7^5$.
\end{example}

The following lemma generalizes Example \ref{ex:inadmissible}.

\begin{lemma} \label{lem: multiplicity} 
    If $P \in \APS(m, n)$ with $P = \{ \xi^{a_1}(p_1) \prec \xi^{a_2}(p_2) \prec \dots \prec \xi^{a_d}(p_d)\}$ then $p_i \neq p_j$ for all $i \neq j$.
\end{lemma}
\begin{proof}
For a contradiction, suppose that $p_i = p_j$ for some $i\neq j$. Then since $P$ is admissible there exists $w \in \Z_m \wr S_n$ with $\Pin(w) = P$. Thus, $w(x) =\xi^{a_i}(p_i)$ and $w(y) =\xi^{a_j}(p_j)$ for some $x\neq y \in [n]$, which implies $w(\xi^{m-a_i}x)=p_i=w(\xi^{m-a_j}y)$, contradicting that fact that $w$ is a bijection.
\end{proof}

\subsection{Characterizing Admissible Pinnacle sets for $\Z_m\wr S_n$} \label{subsec:characterization}
In order to give a complete characterization of the admissible pinnacle sets it will be useful to subdivide any $P \in \APS(m,n)$ as follows. 
For any $0\leq i \leq m-1$, let $\pi_i(P) := P \cap \xi^i[n]$ so that 
\begin{align}
    P &=\bigcup_{i=0}^{m-1}\pi_i(P).
\end{align}
Additionally, since any $P \in \APS(m,n)$ admits potentially many witness permutations, we will often select a unique representative from the set of all possible witnesses in which the pinnacle and non-pinnacle values appear in increasing order when read left to right. 

\par For our first characterization, utilizing the periodicity condition referenced earlier we can cap the length of each admissible pinnacle set quite nicely. This result directly coincides with that of \cite{PinSignedPermutations} and \cite{PinnaclesTypeA}.

\begin{lemma} \label{lem:maxLength}
    For any $w \in \Z_m\wr S_n$, $\#Pin(w) \leq \floor{\frac{n-1}{2}}$. Thus, $\APS_d(m, n) \setminus \APS_{\floor{\frac{n-1}{2}}}(m, n) =\emptyset$ for any $d>\floor{\frac{n-1}{2}}$.
\end{lemma}

\begin{proof}
Suppose $\sigma$ is a generalized permutation in $\Z_m\wr S_n$ with exactly $d$ pinnacles. Then, since any pinnacle $\sigma(x_i)$ of $\sigma$ must satisfy $\sigma(x_{i-1}) \prec \sigma(x_i) \succ \sigma(x_{i+1})$, then $w$ must have at least $d+1$ non pinnacle values. Hence, $2d+1\leq n$. 

To see the bound is tight, for $m>1$ consider $w \in \Z_m \wr S_n$ given by
   \[
    w = \begin{cases}
    \xi^{m-1}(n)\; \xi^0(n-1)\; \xi^{m-1}(n-2)\; \dots\; \xi^{0}(2)\; \xi^{m-1}(1) &; n \text{ even,} \\
    \xi^{m-1}(n)\; \xi^0(n-1)\; \xi^{m-1}(n-2)\; \dots\;\xi^{0}(3) \;\xi^{m-1}(2)\; \xi^{m-1}(1) &; n \text{ odd. }
        \end{cases}
   \]
    Then, $\#\Pin(w) = \floor{\frac{n-1}{2}}$.
The case when $m=1$ follows from \cite[Lemma 2.1]{PinnaclesTypeA}.    
\end{proof}

As a consequence of Lemma \ref{lem:maxLength} the filtration in \eqref{eq:APSfiltration} becomes:
\begin{equation}\label{eq:APSfiltration2}
\APS_0(m,n) \subset \APS_1(m,n) \subset \dots \subset \APS_{\floor{\frac{n-1}{2}}}(m,n)=\APS(m,n).
\end{equation}

In fact, we will see that as consequence of Theorem \ref{thm:pinn-formula} $\APS_d(m,n)$ will in fact be nonempty for all $0\leq d \leq \floor{\frac{n-1}{2}}$.

\begin{definition} \label{def:witness}
Suppose $P = \{ \xi^{a_1}(p_1) \prec \dots \prec \xi^{a_d}(p_d)\}$ is an admissible pinnacle set for $\Z_m \wr S_n$, with $[n]\setminus \{p_1,\dots,p_d\}= \{v_1\prec \dots\prec v_{n-d}\}$. The \newword{canonical witness} $\omega_P \in \Z_m \wr S_n$ of $P$ is the generalized permutation such that $\omega_P(j) := \xi^{m-1}(v_{n-d-j+1})$ for $1\leq j \leq n-2d$ and 
\[
\omega_P(n-2i+1) := \xi^{a_{i}}(p_{i}) 
\qquad \text{and} \qquad 
\omega_P(n-2i+2):= \xi^{m-1}(v_{i})
\]
for each $1\leq i \leq d$. That is, $\omega_P$ has the form, 
\begin{equation*}
\begin{matrix}
    & \xi^{a_1}(p_1) && \xi^{a_2}(p_2) & \dots && \xi^{a_d}(p_d)  
    \\
    \xi^{m-1}(v_{1}) &&\xi^{m-1}(v_{2}) && \dots & \xi^{m-1}(v_{d}) && \xi^{m-1}(v_{d+1}) \dots \xi^{m-1}(v_{n-d}).      
\end{matrix}
\end{equation*}
\end{definition}

It is important to note that in the canonical witness both the pinnacle values and non-pinnacle values appear in increasing order from left to right. 

\begin{example}
    Continuing from Example \ref{ex:noncanonical witness} with $P =\{\xi^1(3)\prec \xi^0(5)\prec \xi^0(2)\} \in \APS_3(3,10)$, in Figure \ref{fig:canonicalwitness} we see that the canonical witness for $P$ is given by
\[\omega_P =
\begin{matrix}
&\xi^1(3)&&\xi^0(5)&&\xi^0(2)&&&&\\
\xi^2(10)&&\xi^2(9)&&\xi^2(8)&&\xi^2(7)&\xi^2(6)&\xi^2(4)&\xi^2(1).
\end{matrix}
\] 
\end{example}

\begin{example}
From Example \ref{ex:witness} we know that $P = \{\xi^4(3)\prec \xi^3(5)\prec \xi^2(2)\} \in \APS(4,10)$ with witness permutation $w = \xi^4(7)\;\xi^2(2)\;\xi^4(6)\;\xi^4(3)\;\xi^4(4)\;\xi^3(5)\;\xi^4(1)\;\xi^4(8)\; \xi^4(9)\; \xi^4(10)$. Observe that when reading the pinnacles (resp. non-pinnacles) of $w$ from left to right they do not appear in increasing order. 

On the other hand, the canonical witness of $P$ is the permutation $\omega_P \in \Z_4 \wr S_{10}$,
\[\omega_P=\;\begin{matrix}
   &\xi^4(3)&&\xi^3(5)&&\xi^2(2)&
\\
\xi^4(10)&&\xi^4(9)&&\xi^4(8)&&\xi^4(7)& \xi^4(6)& \xi^4(4)& \xi^4(1),
\end{matrix}\]
where the pinnacles (resp. non-pinnacles) do appear in increasing order from left to right.
\end{example}

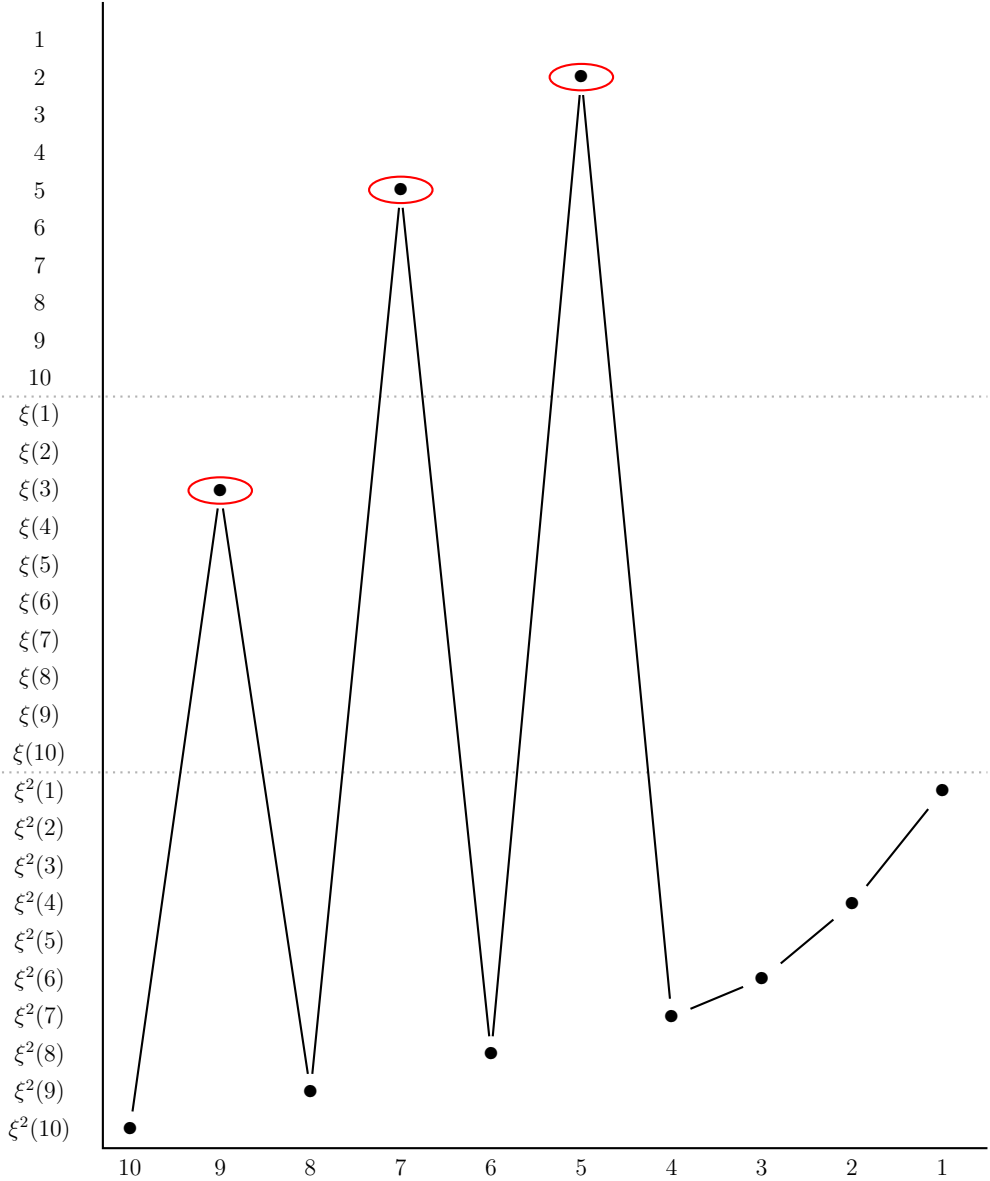
\begin{figure}[ht] 
\begin{tikzpicture}[yscale=.5, xscale=1.2]
%\draw (-1,8) node {(a)};
\draw[black!30, dotted, thick] (-.5,-10.5) -- (10.5,-10.5);
\draw[black!30, dotted, thick] (-.5,-0.5) -- (10.5,-0.5);
\draw[thick] (0.7,10)--(0.7,-20.5)--(10.5,-20.5);
\node at (1,-21)[scale=.75]{$10$};
\node at (2,-21)[scale=.75]{$9$};
\node at (3,-21)[scale=.75]{$8$};
\node at (4,-21)[scale=.75]{$7$};
\node at (5,-21)[scale=.75]{$6$};
\node at (6,-21)[scale=.75]{$5$};
\node at (7,-21)[scale=.75]{$4$};
\node at (8,-21)[scale=.75]{$3$};
\node at (9,-21)[scale=.75]{$2$};
\node at (10,-21)[scale=.75]{$1$};
\node at (0,0)[scale=.75]{$10$};
\node at (0,1)[scale=.75]{$9$};
\node at (0,2)[scale=.75]{$8$};
\node at (0,3)[scale=.75]{$7$};
\node at (0,4)[scale=.75]{$6$};
\node at (0,5)[scale=.75]{$5$};
\node at (0,6)[scale=.75]{$4$};
\node at (0,7)[scale=.75]{$3$};
\node at (0,8)[scale=.75]{$2$};
\node at (0,9)[scale=.75]{$1$};
\begin{scope}[shift={(0,-10)}]
\node at (0,0)[scale=.75]{$\xi(10)$};
\node at (0,1)[scale=.75]{$\xi(9)$};
\node at (0,2)[scale=.75]{$\xi(8)$};
\node at (0,3)[scale=.75]{$\xi(7)$};
\node at (0,4)[scale=.75]{$\xi(6)$};
\node at (0,5)[scale=.75]{$\xi(5)$};
\node at (0,6)[scale=.75]{$\xi(4)$};
\node at (0,7)[scale=.75]{$\xi(3)$};
\node at (0,8)[scale=.75]{$\xi(2)$};
\node at (0,9)[scale=.75]{$\xi(1)$};
\end{scope}
% \node at (0,-1)[scale=.75]{$\xi(1)$};
% \node at (0,-2)[scale=.75]{$\xi(2)$};
% \node at (0,-3)[scale=.75]{$\xi(3)$};
% \node at (0,-4)[scale=.75]{$\xi(4)$};
% \node at (0,-5)[scale=.75]{$\xi(5)$};
% \node at (0,-6)[scale=.75]{$\xi(6)$};
\begin{scope}[shift={(0,-20)}]
\node at (0,0)[scale=.75]{$\xi^2(10)$};
\node at (0,1)[scale=.75]{$\xi^2(9)$};
\node at (0,2)[scale=.75]{$\xi^2(8)$};
\node at (0,3)[scale=.75]{$\xi^2(7)$};
\node at (0,4)[scale=.75]{$\xi^2(6)$};
\node at (0,5)[scale=.75]{$\xi^2(5)$};
\node at (0,6)[scale=.75]{$\xi^2(4)$};
\node at (0,7)[scale=.75]{$\xi^2(3)$};
\node at (0,8)[scale=.75]{$\xi^2(2)$};
\node at (0,9)[scale=.75]{$\xi^2(1)$};
\end{scope}
% \node at (0,-7)[scale=.75]{$\xi^2(1)$};
% \node at (0,-8)[scale=.75]{$\xi^2(2)$};
% \node at (0,-9)[scale=.75]{$\xi^2(3)$};
% \node at (0,-10)[scale=.75]{$\xi^2(4)$};
% \node at (0,-11)[scale=.75]{$\xi^2(5)$};
% \node at (0,-12)[scale=.75]{$\xi^2(6)$};

%\bar{7}\bar{4} \bar{6} 1 \bar{5}2 \bar{3}
\node (a) at (1,-20){$\bullet$};
\node (b) at (2,-3){$\bullet$};
\node (c) at (3,-19){$\bullet$};
\node (d) at (4,5){$\bullet$};
\node (e) at (5,-18){$\bullet$};
\node (f) at (6,8){$\bullet$};
\node (g) at (7,-17){$\bullet$};
\node (h) at (8,-16){$\bullet$};
\node (i) at (9,-14){$\bullet$};
\node (j) at (10,-11){$\bullet$};
\draw[thick] (a)--(b)--(c)--(d)--(e)--(f)--(g)--(h)--(i)--(j);
% \node at (8,8){$\bullet$};
%
\draw[red, thick] (b) circle (10pt);
\draw[red, thick] (d) circle (10pt);
\draw[red, thick] (f) circle (10pt);
%\draw[blue] (4,-10) circle (8pt);
%\draw[blue] (6,-11) circle (8pt);
\end{tikzpicture} 
\caption{The graph for the canonical witness permutation $\omega_P$ for $P =\{\xi^1(3)\prec \xi^0(5)\prec \xi^0(2)\} \in \APS_3(3,10)$ given by
$\omega_P =\xi^2(10)\xi^1(3)\xi^2(9)\xi^0(5)\xi^2(8)\xi^0(2)\xi^2(7)\xi^2(6)\xi^2(4)\xi^2(1) \in \Z_3 \wr S_{10}$. The pinnacle values are circles in red.
}\label{fig:canonicalwitness}
\end{figure}

Naturally, we then need to prove that this witness $\omega_P$ has the property that $Pin(\omega_P) = S$ for any given $S\in \APS(m,n)$.

\begin{lemma}
For any given $P \in \APS_d(m, n)$, the canonical witness $\omega_P \in \Z_m \wr S_n$ satisfies $Pin(\omega_P) = P$. Thus, the canonical witness is a witness. 
\end{lemma}

\begin{proof}
If $P = \{ \xi^{a_1}(p_1) \prec \dots \prec \xi^{a_d}(p_d)\}$ and  $[n]\setminus \{p_1,\dots,p_d\}= \{v_1\prec \dots\prec v_{n-d}\}$, to see that $\omega_P$ has the form given in Definition \ref{def:witness} we first note that by Lemma \ref{lem:maxLength} $d<n-d$, so that the permutation is well-defined. 

Now, to see that $\Pin(\omega_P) =P$ it suffices to prove that $\xi^{m-1}(v_i) \prec \xi^{a_i}(p_i) \succ \xi^{m-1}(v_{i+1})$ for all $1\leq i \leq d$. Evidently, if $a_i \neq m-1$ then the inequalities hold. 
So suppose $a_i = m-1$. Then, since $\xi^{a_j}(p_j)\prec \xi^{a_i}(p_i)$ for all $1\leq j<i$ and $P$ is admissible, then there must exist at least $i+1$ values $\xi^{m-1}(v_1),\dots,\xi^{m-1}(v_{i+1})\in \xi^{m-1}[n]$ that satisfy the desired inequalities $\xi^{m-1}(v_j) \prec \xi^{m-1}(p_j) \succ \xi^{m-1}(v_{j+1})$ for each $ 1 \leq j \leq i$. Thus, $\Pin(\omega_P) =P$.
\end{proof}

In analogy to Definition 2.4 of \cite{PinSignedPermutations} we consider the following special sets.

\begin{definition}
   An subset $S \subseteq \I_n^m$ is \newword{colored} if $S = \pi_i(S)$ for some $i \in \{0,\dots,m-1\}$.
\end{definition}

\begin{proposition}\label{prop:colored}
    Every colored set $S \subseteq \I_n^m$ is an admissible pinnacle set in $\Z_m \wr S_N$ for some integer $N$. 
\end{proposition}

\begin{proof}
Suppose $P = \{ \xi^{a}(p_1) \prec \dots \prec \xi^{a}(p_d)\}$ for some $a\in \{0,\dots,m-1\}$. Let $N=2\ell+1$ where $\ell = \min(p_1,\dots,p_d)$ and set $\{v_1\prec \dots\prec v_{N-d}\}=[N]\setminus \{p_1,\dots,p_d\}$. As in Definition \ref{def:witness} we can form a permutation, 
\[
w = \xi^{m-1}(v_1) \xi^{a}(p_1) \xi^{m-1}(v_2) \dots \xi^{m-1}(v_d)\xi^{a}(p_d)\xi^{m-1}(v_{d+1}) \dots \xi^{m-1}(v_{N-d}).
\]
Evidently, if $a>m-1$ then $\Pin(w) = P,$ so suppose $a=m-1$. Since $\{ p_1,\dots,p_d\} \subseteq \{1,\dots,\ell\}$, then $d\leq \ell$ in which case $N-d\leq \ell+1$. Hence, the set $\{v_1,\dots,v_{d+1}\} \subseteq \{\ell+1,\dots,N\}$ which implies that $v_j<p_i$ for all $1\leq j \leq d+1$ and $1\leq i \leq d$. Thus, $\xi^{m-1}(v_i) \prec \xi^{m-1}(p_i) \succ \xi^{m-1}(v_{i+1})$ for all $1\leq i \leq d$, so $\Pin(w) = P$.
\end{proof}

\begin{example}
    Let $P = \{\xi^2 (6),\xi^2(2), \xi^2(3)  \} \subset \I^4_6$. Following the construction in the proof of Proposition \ref{prop:colored}, we can take $\ell = 6$ so that $N=13$. Thus, setting
    \[V = \xi^3( [13]\setminus \{6,3,2\})=\{\xi^3(13),..., \xi^3(7), \xi^3(5),\xi^3(4),\xi^3(1)\}\]
    we obtain the permutation $w \in \Z_5\wr S_{13}$ with $\Pin(w) = P$ given by, 
    \[w=
    \begin{matrix} 
    &\xi^2(6)&&\xi^2(2)&&\xi^2(3)&
\\
\xi^3(13)&&\xi^3(12)&&\xi^3(11)&&\xi^3(10)\; \dots \;\xi^3(1).
    \end{matrix}
    \]
\end{example}

\begin{corollary}\label{cor:signed}
For any $1\leq i < m-1$ and $P \subset \xi^i[n]$, if $\#P \leq \lfloor \frac{n-1}{2} \rfloor$ then $P \in \APS_{\#P}(m,n)$.
\end{corollary}
\begin{proof}
This follows immediately from the proof in Proposition \ref{prop:colored} by setting $n=N$.
\end{proof}

Let $\mathcal{P}(\I_n^m)$ denote the power set of $\I_n^m$.

\begin{definition} \label{def:psi-map}
For integers $m,n\geq 0$ and $k>0$, let $\xi = e^{\frac{2\pi i}{ m}}$ and $\zeta = e^{\frac{2\pi i}{m+k}}$ be primitive $m^{th}$ and $(m+k)^{th}$ roots of unity, respectively. The assignment $\xi^{a}(x) \mapsto \zeta^{a+k}(x)$ induces two maps:
\begin{itemize}
\item[(I)] The first is the map of groups,  
\[\psi_k: \Z_{m} \wr S_n \longrightarrow \Z_{m+k }\wr S_n\]
which acts on a permutation $w$ by sending each $w(i)=\xi^a(x)$ to $\Psi_k(w(i)):= \zeta^{a+k}(x)$ and extending multiplicatively so that
\[
w=w(n)\dots w(1) \longmapsto \Psi_k(w(n)) \dots \Psi_k(w(1))=: \Psi_k(w).
\]
\item[(II)] The second is the map of power sets,
\[
\Psi_k: \mathcal{P}(\I_n^m) \longrightarrow \mathcal{P}(\I_n^{m+k})
\]
which sends a set $S \subset \I_n^m$ with elements $\{ \xi^{a_i}(x_i)\}_{i \in I}$ to the set $\Psi_k(S) \subset \I_n^{m+k}$ comprised of elements $\{ \zeta^{a_i+k}(x_i)\}_{i \in I}$.
\[
S = \{ \xi^{a_i}(x_i)\}_{i \in I} \longmapsto \{ \zeta^{a_i+k}(x_i)\}_{i \in I}=: \Psi_k(S)
\]
\end{itemize}

\begin{remark}
It is important to note that the map $\psi_k$ in Definition \ref{def:psi-map} is not group homomorphism. Indeed, given $w, \sigma$ in $\Z_m\wr S_n$ sending $w(i)=\xi^{a_i}(x_i)$ and $\sigma(j)=\xi^{b_j}(i_j)$ then for each $j \in [n]$:
\[\psi_k(w \circ \sigma)(j) = \zeta^{b_j+a_{i_j}+k}(x_{i_j})\neq \zeta^{b_j+a_{i_j}+2k}(x_{i_j}) =(\psi_k(w) \circ \psi_k(\sigma))(j).\] 
\end{remark}
\end{definition}

While this property would be nice to have, we do not need it for the purposes of this paper. Far more importantly, $\psi_k$ is injective, and as we will see shortly, \emph{preserves} canonical witnesses.

\begin{lemma} \label{lem: imageSize}
    For any subset $S$, we have $\#S = \# \Psi_k(S)$.
\end{lemma}

\begin{proof}
Since evidently for any $x\neq y \in [n]$ and $a\neq b \in \{0,\dots,m-1\}$ we have that $\zeta^{a+k}(x) \neq \zeta^{b+k}(y)$, then necessarily $S$ and $\Psi_k(S)$ have the same cardinality.
\end{proof}

\begin{lemma}\label{lem:injective}
For any nonnegative integer $k$, the map $\Psi_k$ is injective. Thus, $\psi_k$ is also injective. 
\end{lemma}

\begin{proof}
    Suppose $S,S' \subseteq \I^m_n$ are two subsets such that $S =\{ \xi^{a_1}(x_1)\prec \dots \prec \xi^{a_d}(x_d)\}$ and $S' =\{ \xi^{b_1}(y_1)\prec \dots \prec \xi^{b_e}(y_e)\}$ with $\Psi_k(S) = \Psi_k(S')$. Then by Lemma \ref{lem: imageSize} then $e=d$. Thus, $\zeta^{a_i+k}(x_i) = \zeta^{b_i+k}(y_i)$ for all $1\leq i \leq d$, which implies that $x_i = y_i$ and $a_i = b_i$ for all $i$; consequently, $S=S'$.
    Since $\psi_k(w)$ is defined pointwise on each $w(i)$ for any $w \in Z_m \wr S_n$, then the second claim follows. 
\end{proof}

\begin{theorem} \label{thm: witnessBiject}
    For any $k \geq 0$ the map $\psi_k: \Z_m \wr S_n \to \Z_{m+k}\wr S_n$ preserves witnesses. That is, $\psi_k(w)$ is a witness for $\Psi_k(P)$ for any witness $w$ of $P \in \APS(m,n)$. In particular, $\psi_k(\omega_P) = \omega_{\Psi_k(P)}$.
\end{theorem}

\begin{proof}
Suppose $P \in \APS(m,n)$ with  witness $w$ and pinnacle set $\{\xi^{a_1}(x_1) , \dots ,\xi^{a_d}(x_d)\}$. Then, since $\Psi_k(P)=\{ \zeta^{a_1+k}(x_1), \dots , \zeta^{a_d+k}(x_d)\}$ and all powers of $\zeta$ are increased uniformly, so that any non-pinnacle values $\xi^{b_i}(y_i) \prec \xi^{a_i}(x_i) \succ \xi^{b_{i+1}}(y_{i+1})$ are mapped to $\zeta^{b_i+k}(y_i) \prec \zeta^{a_i+k}(x_i) \succ \zeta^{b_{i+1}+k}(y_{i+1})$, then $\psi_k(w)$ is indeed a witness for $\Psi_k(P)$. 

In the special case when $w = \omega_P$ as in Definition \ref{def:witness}, then since $\zeta^{a_{i}+k}(p_{i})\prec \zeta^{a_{i+1}+k}(p_{i+1})$ for all $1\leq i <d$ and $\zeta^{m-1+k}(v_j) \prec \zeta^{m-1+k}(v_{j+1})$ for all $1\leq j <n-d$, it follows that $\psi_k(\omega_P)$ is in fact the canonical witness for $\Psi_k(P)$.
\end{proof}

\begin{example}
Let $\xi=e^{2 \pi i / 5}$ and $\zeta = e^{2\pi i/8}$. Consider $P = \{\xi^4(3), \xi^3(2)\} \in \APS(5, 5)$ with canonical witness $\omega_P = \xi^4(5)\xi^4(3)\xi^4(4)\xi^3(2)\xi^4(1)$. Then if $k=3$, the map $\psi_3$ sends $\omega_P$ to the permutation in $\Z_8 \wr S_5$ given by
\[
\psi_3(\omega_P) = \zeta^7(5)\zeta^7(3)\zeta^7(4)\zeta^6(2)\zeta^7(1).
\]
As expected $\psi_3(\omega_P)= \omega_{\Psi_3(P)}$. 

\end{example}

\begin{corollary}\label{cor: APSimage}
For any nonnegative integer $k$, 
$\Psi_k(\APS_d(m,n)) =\APS_d(m+k,n) \cap \mathcal{P}(\I_n^{m+k}\setminus \I_n^{k})$. 
\end{corollary}

\begin{proof}
By construction for any $P \in \APS(m,n)$ we have that $\Psi_k(P) \subset \bigcup_{i=k}^{m+k-1} \zeta^i[n]$. Moreover, since by Theorem \ref{thm: witnessBiject} we also have that $\psi_{k}(\omega_P) = \omega_{\Psi_k(P)}$, then $\Psi_k(P)$ is admissible in $\Z_{m+k}\wr S_n$ with witness $\psi_{k}(\omega_P)$. 
\end{proof}

\begin{corollary} \label{cor: APSFiltrationM}
 For any positive integers $m,n,k,d$ the map $\Psi_k$ induces the following filtration on $\APS_d(m,n)$, 
 \[
\Psi_{m-1}(\APS_d(1, n)) \subset\Psi_{m-2}(\APS_d(2, n)) \subset \dots \subset\Psi_1(\APS_d(m-1, n)) \subset \APS_d(m, n).
\]
\end{corollary}

Recall that by \cite[Lemma 3.8]{PinSignedPermutations} for any totally ordered set $X$ and any nonnegative integer $d$, we can define
\begin{equation}\label{eq:APS=set}
\APS_d(X) = \APS_d([\#X]) := \APS_d(1,\#X)
\end{equation}
so that admissibility fundamentally depends only on the cardinality of the underlying set.

We are now ready to give a complete characterization of admissible pinnacle sets for the generalized symmetric group $\Z_m \wr S_n$.

\begin{theorem}\label{thm: pinCharacter}
    A set $P=\bigcup_{i=0}^{m-1}\pi_i(P) \subseteq \I_n^m$ is contained in $\APS_d(m,n)$ 
    if and only if:
    \begin{enumerate}
        \item $d\leq \floor{\frac{n-1}{2}}$,
        \item $|\pi_i(P)| \bigcap |\pi_j(P)| =  \emptyset  \ \forall j \neq i \in [m-1]$, and
    \item[(3)] $P\setminus \pi_0(P) \in \Psi_1\Big(\APS_{d-\#\pi_0(P)}\Big(m-1,n-\#\pi_0(P)\Big)\Big)$.
    \end{enumerate}
\end{theorem}
\begin{proof}
Suppose $P \in \APS_d(m,n)$. Then by Lemma \ref{lem:maxLength} and Lemma \ref{lem: multiplicity} conditions (1) and (2) hold. 
To see condition (3) is also true suppose $P$ has canonical witness, $\omega_P$ as in Definition \ref{def:witness}. Since the pinnacles of $P$ appear in increasing order from left to right in $\omega_P$, then removing the values in $\pi_0(P)$ from $\omega_P$ yields a permutation $\sigma$ on the set $\bigcup_{j=1}^{m-1} \xi^{j}([n]\setminus \pi_0(P))$ with $\Pin(\sigma)=P\setminus \pi_0(P)$.

Hence, by Corollary \ref{cor: APSimage} and Equation \ref{eq:APS=set},
\[P\setminus \pi_0(P) \in \APS_{d-\#\pi_0(P)}\left( \bigcup_{j=1}^{m-1} \xi^{j}\Big([n]\setminus \pi_0(P)\Big)\right)=\Psi_1\Big(\APS_{d-\#\pi_0(P)}\Big(m-1,n-\#\pi_0(P)\Big)\Big).\]

Now assume that $P=\bigcup_{i=0}^{m-1}\pi_i(P) \subseteq \I_n^m$ satisfies conditions (1)-(3).
Then if $P = \{ \xi^{a_1}(p_1)\prec \dots \prec \xi^{d}(p_d)\}$ and $P\setminus \pi_0(P) = \{ \xi^{a_1}(p_1)\prec \dots \prec \xi^{\ell}(p_{\ell})\}$ for some $\ell\leq d$, by condition (3) $P\setminus \pi_0(P)$ must have a canonical witness of the form, 
\begin{equation}\label{eq:smallwitness} 
\omega_{P\setminus\pi_0(P)}=\begin{matrix}&\xi^{a_1}(p_1)&&\xi^{a_\ell}(p_{\ell})\\
\xi^{m-1}(v_1)&& \dots \xi^{m-1}(v_{\ell})&& \xi^{m-1}(v_{\ell+1})\dots \xi^{m-1}(v_{n-d})
\end{matrix}. 
\end{equation}
Since by (1) $n-d-\ell > d-(\ell-1)$ and evidently, any $\xi^0(p_i) \in \pi_0(P)$ satisfies $\xi^{m-1}(v_i) \prec \xi^0(p_i) \succ \xi^{m-1}(v_{i+1})$ for any $\ell+1\leq i \leq d$, then the elements in $\pi_0(P)$ can be inserted into the tail of $\omega_{P\setminus\pi_0(P)}$ as follows: 
\begin{equation*}
\omega_{P\setminus\pi_0(P)}\sqcup \pi_0(P)=\begin{matrix}&&\xi^{0}(p_\ell)&&\xi^{0}(p_{d})
\\
\dots&\xi^{m-1}(v_{\ell+1})&& \dots \xi^{m-1}(v_{d})&& \xi^{m-1}(v_{d+1})\dots \xi^{m-1}(v_{n-d})
\end{matrix}. 
\end{equation*}
Since by (2) none of the pinnacle values are repeated, then in fact $\Pin\Big(\omega_{P\setminus\pi_0(P)}\sqcup \pi_0(P)\Big) =P$ with
$\omega_{P\setminus\pi_0(P)}\sqcup \pi_0(P) = \omega_P$ so that $P \in \APS_d(m,n)$.

\end{proof}

It turns out that there is a symmetry between the necessary condition in Theorem \ref{thm: pinCharacter} for the subset $\pi_0(P)$ and the subset $\pi_{m-1}(P)$, as either of these determine the admissibility of the set.

\begin{corollary}\label{cor:pinCharacter}
    A set $P=\bigcup_{i=0}^{m-1}\pi_i(P) \subseteq \I_n^m$ is contained in $\APS_d(m,n)$ 
    if and only if:
    \begin{enumerate}
        \item $d\leq \floor{\frac{n-1}{2}}$,
        \item $|\pi_i(P)| \bigcap |\pi_j(P)| =  \emptyset  \ \forall j \neq i \in [m-1]$, and
        \item $\pi_{m-1}(P) \in \Psi_{m-1}\left(\APS_{\#\pi_{m-1}}\left(1, n- \sum_{j\neq m-1}\# \pi_j(P)\right)\right).$
    \end{enumerate}
\end{corollary}
\begin{proof}
Suppose $P \in \APS(m,n)$. In light of Theorem \ref{thm: pinCharacter}, it suffices to prove (3) holds.   
So let $\omega_P$ as in Definition \ref{def:witness}. If $\pi_{m-1}(P)=\emptyset$ then (3) is trivially satisfied. So suppose not and let $d_{m-1}=\#\pi_{m-1}(P)$ so that $a_k=m-1$ for all $1\leq k\leq d_{m-1}$. Then the subword of $\omega_P$ obtained by omitting all values $\xi^{a_i}(p_i)$ for $i>d_{m-1}$ given by
\begin{equation*} 
\begin{matrix}&\xi^{m-1}(p_1)&&\xi^{m-1}(p_{d_{m-1}})\\
\xi^{m-1}(v_1)&& \dots \xi^{m-1}(v_{d_{m-1}})&& \xi^{m-1}(v_{d_{m-1}+1})\dots \xi^{m-1}(v_{n-d})
\end{matrix} 
\end{equation*}
is precisely the canonical witness $\omega_{\pi_{m-1}(P)}$ for $\pi_{m-1}(P)$ in  $\APS_d\Big(\xi^{m-1}[n] \setminus \bigcup_{j\neq i}\xi^{m-1}|\pi_j(P)|\Big)$. From here (3) follows since by Corollary \ref{cor: APSimage} and Equation \eqref{eq:APS=set} we have that,
\begin{align*}
\APS_{d_{m-1}}\left(\xi^{m-1}[n] \setminus \bigcup_{j\neq m-1}\xi^{m-1}|\pi_j(P)|\right) &= \Psi_{m-1}\left(\APS_{d_{m-1}}\left([n] \setminus \bigcup_{j\neq m-1}|\pi_j(P)|\right)\right)\\
&=\Psi_{m-1}\left(\APS_{d_{m-1}}\left(1, n- \sum_{j\neq m-1}\# \pi_j(P)\right)\right).
\end{align*}
For the reverse direction suppose (1)-(3) holds. By (3) we have that the canonical witness of $\pi_{m-1}(P)$ is given by Equation \eqref{eq:smallwitness}. So then, suppose for each $j \neq m-1$ we have $\pi_k(P) = \{\xi^k(p_{k,1})\prec \dots \xi^k(p_{k,d_k})\}$ with $d_k = \# \pi_k(P)$. Since by (1), $2d = \sum_i 2d_i \leq n-1$ then we can insert the remaining values $\bigcup_{j\neq m-1} \pi_j(P)$ into $\omega_{\pi_{m-1}(P)}$ after $\xi^{m-1}(v_{d_{m-1}+1})$ and obtain the canonical witness permutation $\omega_{P}$. Specifically, the values in the set $\pi_k(P)$ can be inserted between every $\xi^{m-1}(v_\ell)$ and $\xi^{m-1}(v_{\ell+1})$ for $\ell \in \{ 1+\sum_{j=k}^{m-1} d_j, \dots, 1+\sum_{j=k+1}^{m-1} d_j\}$ as follows,
\[
\begin{matrix}&\xi^{k}(p_{k,1})&&&\xi^{k}(p_{k,d_{k}})
\\
\dots \xi^{m-1}(v_{1+\sum_{j=k}^{m-1} d_j})&& \dots & \xi^{m-1}(v_{\sum_{j=k+1}^{m-1} d_j})&& \xi^{m-1}(v_{1+\sum_{j=k+1}^{m-1} d_j})\dots 
\end{matrix}
\]
Since by (2) the elements $p_{i,j}$ are all distinct, then the construction above yields a well defined permutation (in fact $\omega_P$) in $\Z_m \wr S_n$ with pinnacle set $P$. Thus, $P \in \APS_d(m,n)$.

\end{proof}

\subsection{Enumerating Pinnacle Sets for Generalized Symmetric Groups}\label{subsec:enumeration}

Given the complete characterization of the admissible pinnacle sets for $\Z_m \wr S_n$ above, we are now ready to compute the cardinalities of these sets. In this section we give four distinct formulas for these values, two recursive and two closed. To that effect we define the following values.

\begin{definition}
For any integers $m,n>0$ and $0 \leq d \leq \floor{\frac{n-1}{2}}$, define the function \[\p_{m,n}(d):=\#\APS_d(m,n).\] 
\end{definition}

The first recursion we provide is a generalization of \cite[Proposition 3.11]{PinSignedPermutations} (where $m=1$) which counts the number of admissible pinnacle sets in $\Z_m\wr S_n$ in terms of those in $\Z_{m-1} \wr S_i$ as $i$ ranges over $\{n-d,\dots,n\}$.

\begin{theorem}\label{thm:recursion}
For any $m,n\geq 1$ and $0 \leq d \leq \floor{\frac{n-1}{2}}$,
the function $\p_{m,n}(d)$ satisfies the recursion
\[
\p_{m, n}(d) =\sum_{i=0}^d\binom{n}{i}\p_{m-1, n-i}(d-i)
\]
where 
$\p_{1, n}(d)=\binom{n-1}{d}$
and $\p_{m,n}(0)=1$.
\end{theorem}
\begin{proof}
We proceed by induction on $m$. If $m=1$, then by \cite[Proposition 3.11]{PinSignedPermutations} $\p_{1,n}(d) =\binom{n-1}{d}$ as claimed. 

So suppose the recursion holds for all $\p_{k,n}(d)$ with $k<m$. By Lemma \ref{lem: imageSize} we know that for all $\ell\geq 0$ 
\[ \# \Psi_1(\APS_\ell(m-1,n))=\p_{m-1,n}(\ell).\]
Moreover, by Corollary \ref{cor: APSFiltrationM} we know that $\Psi_1(\APS_d(m-1,n)) \subset \APS_d(m,n)$ so 
\[\#(\APS_d(m,n)\setminus \Psi_1(\APS_d(m-1,n))) = \p_{m,n}(d) - \p_{m-1,n}(d).\]
Similarly, we also know that by Corollary \ref{cor: APSimage} the set $\Psi_1(\APS_d(m-1,n))$ consists of all subsets $S \in APS(m,n)$ with $\pi_0(S) = \emptyset$. Thus, a set $P$ is contained in $\APS_d(m,n)\setminus \Psi_1(\APS_d(m-1,n))$ if and only if $\pi_0(P) \neq \emptyset$. So let $\#\pi_0(P)=i$ and recall that by Theorem \ref{thm: pinCharacter} we have that 
\[P\setminus \pi_0(P) \in \Psi_1\left(\APS_{d-i}(m-1,n-i)\right).\]
Then, by the induction hypothesis and Lemma \ref{lem: imageSize}, we obtain that
\[
\# \Psi_1(\APS_{d-i}(m-1,n-i)) = \p_{m-1,n-i}(d-i).
\]
Since for each for some $1\leq i \leq d$ there are exactly $\binom{n}{i}$ ways to choose the pinnacles in $\pi_0(P)$, then there are precisely
\[
 \sum_{i=0}^d\binom{n}{i}\p_{m-1, n-i}(d-i)
\]
subsets $P$ in $\APS_d(m,n)\setminus \Psi_1(\APS_d(m-1,n))$, from which the result follows.
 \end{proof}

\begin{table}[ht]
    \centering
    \begin{tabular}{|c||*{13}{c|}c||}
        \hline
        m, n & 3 & 4 & 5 & 6 & 7 & 8 & 9 & 10 & 11 & 12\\
        \hline\hline
        $\#\APS(1, n)$ & 2 & 3 & 6 & 10 & 20 & 35 & 70 & 126 & 252 & 462\\
        \hline
        $\#\APS(2, n)$ & 5 & 7 & 31 & 49 & 209 & 351 & 1471 & 2561 & 10625 & 18943\\
        \hline 
        $\#\APS(3, n)$ & 8 & 11 & 76 & 118 & 776 & 1283 & 8236 & 14146 & 89528 & 157742\\
        \hline
        $\#\APS(4, n)$ & 11 & 15 & 141 & 217 & 1931 & 3167 & 27421 & 46761 & 398331 & 697359\\ 
        \hline 
        $\#\APS(5, n)$ & 14 & 19 & 226 & 346 & 3884 & 6339 & 69106 & 117326 & 1256804 & 2191534 \\
        \hline
        $\#\APS(6, n)$ & 17 & 23 & 331 & 505 & 6845 & 11135 & 146395 & 247801 & 3198557 & 5562287 \\ 
        \hline
        $\#\APS(7, n)$ & 20 & 27 & 456 & 694 & 11024 & 17891 & 275416 & 465186 & 7026480 & 12194958\\ 
        \hline
        $\#\APS(8, n)$ & 32 & 31 & 601 & 913 & 16631 & 26943 & 475321 & 801521 & 13868183 & 24033247 \\ 
        \hline 
        $\#\APS(9, n)$ & 26 & 35 & 766 & 1162 & 23876 & 38627 & 768286 & 1293886 & 25231436 & 43674254\\ 
        \hline  
        $\#\APS(10, n)$ & 29 & 39 & 951 & 1441 & 32969 & 53279 & 1179511 & 1984401 & 43059609 & 74463519 \\
        \hline
    \end{tabular} 
    \caption{Cardinalities of $\APS(m, n)$ for all $1 \leq m \leq 10$, and $3 \leq n \leq 12$.}
\end{table}

It turns out that the cardinality of $\APS_d(m,n)$ can also be expressed beautifully as the first $d+1$ terms of $(-1)^d(m-1)^n$. This gives us the first closed formula for $\p_{m,n}(d)$.

\begin{theorem}\label{thm:pinn-formula}
    For any integers $m,n>0$ and $0 \leq d \leq \floor{\frac{n-1}{2}}$, we have that 
\[
\p_{m,n}(d)= \sum_{i=0}^{d} \binom{n}{i}m^i(-1)^{i+d}.
\]
\end{theorem}

\begin{proof}
    Evidently $ \sum_{i=0}^{d} \binom{n}{i}(-1)^{i+d} = \binom{n-1}{d}=\p_{1,n}(d)$. We proceed by induction on $m$. 

    Suppose that $\p_{m-1,\ell}(d)= \sum_{i=0}^{d} \binom{\ell}{i}(m-1)^i(-1)^{i+d}$ for all positive integers $\ell$ and $d\leq \floor{(\ell-1)/2}$. Then, 
\begin{align}
    \p_{m, n}(d) &=\sum_{k=0}^d\binom{n}{k}\p_{m-1, n-k}(d-k) = \sum_{k=0}^d\binom{n}{k}\sum_{i=0}^{d-k} \binom{n-k}{i}(m-1)^i(-1)^{i+d-k} \nonumber \\
    &=  \sum_{k=0}^d\binom{n}{k}\sum_{i=0}^{d-k} \binom{n-k}{i}(-1)^{i+d-k} \sum_{a =0}^i \binom{i}{a}m^{a}(-1)^{i-a} \nonumber \\
    &= \sum_{a=0}^d m^a \sum_{j=0}^{d-a}\sum_{k=0}^{d-(j+a)} \binom{n}{k}\binom{n-k}{j+a}\binom{j+a}{a}. \label{eq:B}
    \end{align}
    Having factored out $m$, we see that the coefficient of $m^a$ in \eqref{eq:B} is given by,
\begin{align}
    \sum_{j=0}^{d-a}\sum_{k=0}^{d-(j+a)} \binom{n}{k}\binom{n-k}{j+a}\binom{j+a}{a} = \sum_{j=0}^{d-a} \binom{n}{j+a}\binom{j+a}{a} \sum_{k=0}^{d-(j+a)} \binom{n-(j+a)}{j}(-1)^{d+k+a}\nonumber\\
    = \sum_{j=0}^{d-a} \binom{n}{j+a}\binom{j+a}{a}\binom{n-(j+a)-1}{d-(j+a)}(-1)^j = \binom{n}{a} \sum_{j=0}^{d-a} (-1)^j \binom{n-a}{j}\binom{n-(j+a)-1}{d-(j+a)}.\label{eq:D}
\end{align}
So now, since $\binom{n}{j}\binom{n-1-j}{n-1-d}=\frac{n-d}{n-j} \binom{n}{d}\binom{d}{j}$ it follows that, 
\begin{align}\label{eq:A}
    \sum_{j=0}^{d-a} (-1)^j \binom{n-a}{j}\binom{n-(j+a)-1}{d-(j+a)} =(n-d) \binom{n-a}{d-a} \sum_{j=0}^{d-a} \frac{(-1)^j }{n-a-j}\binom{d-a}{j}.
\end{align}
Moreover, since for any $1\leq x \leq d$ we have $\sum_{j=0}^d \frac{(-1)^j}{x+j}\binom{d}{j} = x^{-1}\binom{d+x}{d}^{-1}$ (see \cite[Eq. (5.41)]{ConcreteMathematics}), then sending $j\mapsto d-a-j$ it follows that the sum in \eqref{eq:A} is equal to,
\begin{align}
\sum_{j=0}^{d-a} \frac{(-1)^j }{n-a-j}\binom{d-a}{d-a-j} = 
\sum_{j=0}^{d-a} 
\frac{(-1)^{d-a-j}}{n-d+j}\binom{d-a}{j}
= 
(n-d)^{-1}\binom{n-a}{d-a}^{-1}(-1)^{d-a}\label{eq:C}
\end{align}
Thus, combining Equations \eqref{eq:B}, \eqref{eq:D},\eqref{eq:A}, \eqref{eq:C} we obtain the desired result, 
\begin{align*}
    \p_{m,n}(d)&=\sum_{a=0}^d m^a \sum_{j=0}^{d-a}\sum_{k=0}^{d-(j+a)} \binom{n}{k}\binom{n-k}{j+a}\binom{j+a}{a}\\
    &=\sum_{a=0}^d m^a \binom{n}{a} \sum_{j=0}^{d-a} (-1)^j \binom{n-a}{j}\binom{n-(j+a)-1}{d-(j+a)}\\
    &=\sum_{a=0}^d m^a \binom{n}{a} (-1)^{d+a}.
\end{align*}
\end{proof}

The formula in Theorem \ref{thm:pinn-formula} can also be expressed as a purely positive sum without signs. The following result provides the second closed form for $\p_{m,n}(d)$.

\begin{corollary}\label{cor:pinformula}
For any integers $m,n>0$ and $0 \leq d \leq \floor{\frac{n-1}{2}}$ we have,
\[
\p_{m,n}(d)= \sum_{k=0}^d(m-1)^k \binom{n}{k}\binom{n-k-1}{d-k}.
    \]
\end{corollary}

\begin{proof}
    The proof follows from direct computation. From Theorem \ref{thm:pinn-formula} we obtain:
\begin{align*}
 \p_{m,n}(d)
&= \sum_{i=0}^{d} (-1)^{i+d} \binom{n}{i}m^i
= \sum_{i=0}^{d} (-1)^{i+d} \binom{n}{i}\sum_{k=0}^i \binom{i}{k}(m-1)^k\\
&= \sum_{k=0}^d \sum_{i=k}^{d} (-1)^{i+d} \binom{n}{i} \binom{i}{k} (m-1)^k
= \sum_{k=0}^d \sum_{i=k}^{d} (-1)^{i+d} \binom{n}{k} \binom{n-k}{i-k}(m-1)^k\\
&= \sum_{k=0}^d  (m-1)^k\binom{n}{k} (-1)^{d+k} \sum_{i=0}^{d-k} (-1)^{i} \binom{n-k}{i}
= \sum_{k=0}^d  (m-1)^k\binom{n}{k}  \binom{n-1-k}{d-k}.
\end{align*}
\end{proof}

In \cite{PinSignedPermutations} the set of admissible pinnacle sets for the usual symmetric group $S_n$ and the signed symmetric group $\Z_2 \wr S_n$ were enumerated. Indeed, in \cite[Proposition 3.11]{PinSignedPermutations} and \cite[Theorem 3.12]{PinSignedPermutations} it was shown that
\[
\#\APS(1,n) = \binom{n-1}{\floor{\frac{n}{2}}}
\qquad \text{and} \qquad
\#\APS(2,n) = \sum_{k=0}^{\floor{\frac{n}{2}}}{\binom{n}{k}\binom{n-1-k}{\floor{\frac{n-1}{2}}-k}}.
\]
 Thus, our results also easily recover and fully generalize the results in \cite{PinSignedPermutations}.

\begin{corollary}\label{cor:APS}
    The number of admissible pinnacle sets in $\Z_m \wr S_n$ for any $n\geq 2$ and $m\geq 1$ is given by:
    \[
\#\APS(m,n)= (-1)^{\floor{\frac{n-1}{2}}}\sum_{i=0}^{\floor{\frac{n-1}{2}}} \binom{n}{i}m^i(-1)^{i}= \sum_{k=0}^{\floor{\frac{n-1}{2}}}(m-1)^k \binom{n}{k}\binom{n-k-1}{\floor{\frac{n-1}{2}}-k}.
    \]
In particular, $\#\APS(m,n)$ grows exponentially with $n$ but polynomially with $m$.
\end{corollary}

In \cite[Prop. 3.11]{PinSignedPermutations} it was shown that
$\p_{1, n}(d)=\p_{1,n-1}(d-1) + \p_{1,n-1}(d)=\binom{n-1}{d}$. 
A similar recursion for the case when $m=2$ was also conjectured in \cite[Conjecture 5.1]{PinSignedPermutations}. 
As an easy consequence of Theorem \ref{thm:pinn-formula} we prove and generalize these recursions for all relevant values of $m,n,d$. In particular, the recursion below differs from that in Theorem \ref{thm:recursion} in that it expresses the number of admissible pinnacle sets in $\Z_m \wr S_n$ in terms of those in $\Z_m\wr S_{n-1}$ as opposed to those in $\Z_{m-1} \wr S_i$ for $n-d\leq i \leq n$.

\begin{theorem}\label{thm:recursion2}
    For any positive integers $m,n$ and $d\leq \floor{\frac{n-1}{2}}$ we have:
    \begin{equation}\label{eq:recursion2}
\p_{m,n}(d) = m\;\p_{m,n-1}(d-1)+\p_{m,n-1}(d). \end{equation}
Thus, at $m=2$ \eqref{eq:recursion2} recovers the sequence \cite[A119258]{OEIS}, so that Conjecture 5.1 in \cite{PinSignedPermutations} is true.
\end{theorem}

\begin{proof}
    By Theorem \ref{thm:pinn-formula} we have: 
    \begin{align*}
       \p_{m, n}(d) - \p_{m, n-1}(d) 
        &=  \sum_{i=0}^d (-1)^{i+d}m^i\left\lbrack \binom{n}{i} - \binom{n-1}{i} \right\rbrack  
        = \sum_{i=1}^d (-1)^{i+d}m^i\binom{n-1}{i-1}  \\
       & =m\sum^{d-1}_{i=0}(-1)^{i+d-1}m^i\binom{n-1}{i} =m\;\p_{m, n-1}(d-1).
    \end{align*}
    The remaining claims follow by setting $m=2$ and $d=\floor{\frac{n-1}{2}}$.
\end{proof}

\section{Pinnacles for Complex Reflection Groups}\label{sec:ComplexReflGroup}
In order to extend the results from generalized symmetric groups to complex reflection groups we will realize $G(m,p,n)$ as certain subgroups inside $\Z_m \wr S_n$. To that effect, we define the following. 

\begin{definition}
    Let $w=w(1)\dots w(n)$ be any element in $\Z_m \wr S_n$ so that for each $i$ we can write $w(i)=\xi^{e_i}(x_i)$ for some $e_i \in \{1,\dots,m-1\}$ and $x_i \in [n]$. Then, for each $w$ define $\varepsilon_w$ as the sum of the $\xi$-powers, $\varepsilon_w:= \sum_{i=1}^ne_i$. 
\end{definition}

Suppose $m,n,p>0$ such that $p$ divides $m$. The \newword{complex reflection group $G(m,p,n)$} is the normal subgroup of $\Z_m \wr S_n$ comprised of the generalized permutations $w$ for which $\varepsilon_w\equiv 0 \;(mod \;p)$.

In particular, when $p=1$ we have that $G(m,1,n)=\Z_m \wr S_n$. 
More generally, whenever $q | p| m$ we have the chain of inclusions:
\begin{equation}\label{eq:nestedgroups}
G(m,q,n) \subseteq G(m,p,n) \subseteq G(m,1,n).
\end{equation}
These inclusions imply that all the previous definitions from \S\ref{Sec:GeneralizedSymGroup} for generalized symmetric groups descend to $G(m,p,n)$. In particular, a subset $P \subseteq \I_n^m$ is an \newword{admissible pinnacle set for $G(m,p,n)$} if there exists a witness permutation $w \in G(m,p,n)$ such that $\Pin(w) = P$. 

\begin{remark}Throughout this section we always assume $m,n,p>0$ with $p$ dividing $m$.\end{remark}

\begin{definition}\label{def:APS-ComplexRflGrp}
For any integer $d\geq 0$, let $\APS_d(m,p,n)$ denote the collection of admissible pinnacle sets $P 
\subseteq \I_n^m$ for $G(m,p,n)$ such that $\#P \leq d$, and set $\APS(m,p,n) := \bigcup_d \APS_d(m,p,n)$.
\end{definition}

Notice that in light of \eqref{eq:nestedgroups}, if $q|p|m$ we have that: 
\begin{equation}\label{eq:APSinclusion}\APS_d(m,q,n)\subseteq \APS_d(m,p,n)
\subseteq \APS_d(m,1,n)=\APS_d(m,n).
\end{equation}
Similarly, it follows from Lemma \ref{lem:maxLength} that for $d>\floor{\frac{n-1}{2}}$, $\APS_d(m,p,n)\setminus \APS_{\floor{\frac{n-1}{2}}}(m,p,n) = \emptyset$, thus $\APS(m,p,n)=\APS_{\floor{\frac{n-1}{2}}}(m,p,n)$.

\begin{lemma}\label{lem:epsilon-range}
    Let $P \in \APS(m,n)$. For any witness $w \in \Z_m \wr S_n$ of $P$ we have that $\varepsilon_w \leq \varepsilon_{\omega_P}$. Moreover, if $\varepsilon_w < \varepsilon_{\omega_P}$ for some witness $w$, then for all $\varepsilon_w< \ell <\varepsilon_{\omega_P}$ there exists another witness $w'$ of $P$ such that $\varepsilon_{w'}=\ell$.
\end{lemma}

\begin{proof}
Suppose $P = \{\xi^{a_i}(p_i)\}_{i=1}^d$ for some $a_i \in \{0,\dots,m-1\}$ and let $\{v_j\}_{j=1}^{n-d}=[n]\setminus|P|$. Then, since by construction the non-pinnacle values of $\omega_P$ are of the form $\xi^{m-1}(v_j)$ and any other witness $w$ of $P$ must have non-pinnacle values $\xi^{b_j}(v_j)$ for some $0\leq b_j\leq m-1$, then evidently 
\[\varepsilon_w = \sum a_i + \sum b_j \leq \sum a_i + (n-d)(m-1) = \varepsilon_{\omega_P}.\] 

Now, suppose that $\varepsilon_w < \varepsilon_{\omega_P}$ for some witness $w$ of $P$ with non-pinnacle values $\xi^{b_j}(v_j)$ as above. Then, any nontrivial assignment of the form $\xi^{b_j}(v_j) \mapsto \xi^{c_j}(v_j)$ for any $b_j\leq c_j \leq m-1$ yields another permutation $w'$ with pinnacle set $P$ and $\varepsilon_w < \varepsilon_{w'}<\varepsilon_{\omega_P}$. In particular, the values of $c_j$ can be chosen so that every value between $\varepsilon_w$ and $\varepsilon_{\omega_P}$ is obtained. 
\end{proof}

This result immediately yields a characterization of pinnacle sets in the complement of $\APS(m,p,n)$.

\begin{corollary}
A set $P \in \APS(m,1,n)\setminus\APS(m,p,n)$ if and only if for all $w \in \Z_m \wr S_{n}$ for which $\Pin(w)=P$ there exists no integer $k$ such that $\varepsilon_w \leq pk \leq \varepsilon_{\omega_P}$.
\end{corollary}

\begin{proof}
If $P \in \APS(m,1,n)\setminus\APS(m,p,n)$ then there exists no witness $w$ of $P$ such that $\varepsilon_w =pk$ for any $k>0$. Since by Lemma \ref{lem:epsilon-range} every value between any two witnesses must be achieved, then there cannot exist a witness $w$ satisfying $\varepsilon_w \leq pk \leq \varepsilon_{\omega_P}$ for any $k$. 

Conversely, if for every witness $w$ of $P \in \APS(m,1,n)$ there exists no $k$ such that $\varepsilon_w=pk$ then $P \notin \APS(m,p,n)$.
\end{proof}

As it turns out, the set $\APS(m,1,n)\setminus\APS(m,p,n)$ is actually empty most of the time. 

\begin{theorem} \label{thm:APS-complex-equal}
    For any $0 \leq d  < \ceil{\frac{n-1}{2}}$ we have $\APS_d(m, p, n) = \APS_d(m, 1, n)$.
\end{theorem}

\begin{proof}
    Evidently, by \eqref{eq:APSinclusion} it suffices to prove that $\APS_d(m, 1, n) \subseteq \APS_d(m, p, n)$.

 So suppose $P \in \APS_d(m,1,n)$. If $\varepsilon_{\omega_P}=pk$ we're done. So suppose not and let $k\in \Z_{>0}$ be maximal such that $\varepsilon_{\omega_P} - pk> 0$. Since $d  < \ceil{\frac{n-1}{2}}$, it follows $n-2d\geq 2$. Thus, the canonical witness $\omega_P$ of $P$, with form given as in Definition \ref{def:witness}, satisfies the property that its tail $\xi^{m-1}(v_{d+1}) \dots \xi^{m-1}(v_{n-d})$ has length at least two. So let $w'$ be the permutation obtained from $\omega_P$ by replacing the final term $\xi^{m-1}(v_{n-d})$ with $\xi^{\ell}(v_{n-d})$ where $\ell:=(m-1)+pk -\varepsilon_{\omega_P}$. Since $0 < \varepsilon_{\omega_P}-pk < p \leq m$ then $ \varepsilon_{\omega_P} > pk >\varepsilon_{\omega_P}- (m-1)$ and so $m-1>\ell>0$.
Therefore, $w'$ is a well-defined witness for $P$ with 
$ \varepsilon_{w'} = \varepsilon_{\omega_P} - (m-1)+\ell= pk
$
and thus $P \in \APS_d(m,p,n)$.
\end{proof}

The bound $d  < \ceil{\frac{n-1}{2}}$ in Theorem \ref{thm:APS-complex-equal} highlights an important discrepancy that depends on the parity of $n$. Indeed, when $n$ is even since $\floor{\frac{n-1}{2}}=\ceil{\frac{n-1}{2}}-1$ then the maximal possible value for $d$ is achieved. Thus, we have the following corollary. 

\begin{corollary}\label{cor:APS-even}
    The equality 
   $\APS(m,p,2r)=\APS(m,1,2r)$ holds for all integers $m,p,r>0$.
\end{corollary}

On the other hand when $n$ is odd, since $ \ceil{\frac{n-1}{2}} = \floor{\frac{n-1}{2}}$, then the maximal value $d= \floor{\frac{n-1}{2}}$ is not within the range of Theorem \ref{thm:APS-complex-equal}. This case is much more complex than the previous situation so providing a closed enumeration has not been possible so far. Nonetheless, below we provide a substantial reduction of this problem to the much simpler situation when $m=p$. To explain this simplification we first need some additional properties that occur when $n$ is odd.

\begin{lemma} \label{lem:APScomplement}
    Suppose $P \in \APS(m,1,2r+1)$ and $m>p$. If $P \cap \I^{m-p}_{2r+1} \neq \emptyset$, then $P \in \APS(m,p,2r+1)$. 
\end{lemma}

\begin{proof}
    Suppose the canonical witness $\omega_P$ is as in Definition \ref{def:witness} so that $P = \{\xi^{a_i}(p_i)\}$. If $\xi^{a_i}(p_i) \in P \cap \I^{m-p}_{2r+1} $ for some $i$ then, since by construction $a_1<a_j$ for all $j$, we must have $a_1<m-p$. In particular, this implies that we may construct a new witness $w'$ for $P$ from $\omega_P$ by sending the non-pinnacle value $\xi^{m-1}(v_1) \mapsto \xi^{m-p}(v_1)$ and keeping all other entries the same. Since $\varepsilon_{\omega_P} - \varepsilon_w = p$ then by Lemma \ref{lem:epsilon-range} there must exists a witness $w'$ of $P$ such that 
    $\varepsilon_w\leq \varepsilon_{w'}\leq \varepsilon_{\omega_P}$ with $\varepsilon_{w'}=pk$ for some $k>0$. Thus, $P \in \APS(m,p,2r+1)$ as claimed.
\end{proof}

With this characterization in hand, we prove that the map $\Psi_{m-p}$ maps the complement of 
$\APS(p,p,2r+1)$ is surjective onto the complement of $\APS(m,p,2r+1)$ .

\begin{proposition}\label{prop:odd-reduction1} 
    For all $P \in \APS(m,1,2r+1) \setminus \APS(m,p,2r+1)$ there exists $P'  \in \APS(p,1,2r+1) \setminus \APS(p,p,2r+1)$ such that $P = \Psi_{m-p}(P')$.
\end{proposition}

\begin{proof}
If $m=p$ then $\Psi_{m-p}$ is the identity map and the statement is trivial. So suppose $m>p$.
    By Lemma \ref{lem:APScomplement} if $P \in \APS(m,1,2r+1) \setminus \APS(m,p,2r+1)$ then $P = \{ \xi^{a_i}(p_i)\}_{i=1}^d$ for some $a_i\geq m-p$ for all $i$ where $\xi$ is a primitive $m^{th}$ root of unity. Letting $\zeta$ be any primitive $p^{th}$ root of unity, it follows from Corollary \ref{cor: APSimage} that the set $P' = \{ \zeta^{a_i-(m-p)}(p_i)\}_{i=1}^d$ is an admissible pinnacle set in $\APS(p,1,2r+1)$ for which $\Psi_{m-p}(P')=P$.
 So suppose $P' \in \APS(p,p,2r+1)$ with witness $w \in G(p,p,2r+1)$ so that $\varepsilon_w = pk$ for some $k>0$. Then, since $m=p\ell$ for some $\ell>1$ and by Theorem \ref{thm: witnessBiject} the permutation $\psi_{m-p}(w)$ is a witness for $P$, it follows that
 \[
\varepsilon_{\psi_{m-p}(w)} = \varepsilon_w +(m-p)(2r+1) = pk+p(\ell-1)(2r+1) \equiv 0 (mod \; p).
 \]
 Consequently, $P \in \APS(m,p,2r+1)$ which is a contradiction. Thus, $P'  \notin \APS(p,p,2r+1)$.
\end{proof}

The only thing needed to complete our enumeration is to prove that in fact the map on the complements is a bijection, which we now show.

\begin{theorem}\label{thm:odd-reduction}
For any positive integers $m,p,r$, we have that 
\[\Psi_{m-p}(\APS(p,1,2r+1) \setminus \APS(p,p,2r+1)) = \APS(m,1,2r+1) \setminus \APS(m,p,2r+1).\]
\end{theorem}

\begin{proof}
Once again, if $m=p$ the statement is trivial, so assume $m>p$.
    By Proposition \ref{prop:odd-reduction1} it suffices to prove the inclusion 
    \[\Psi_{m-p}(\APS(p,1,2r+1) \setminus \APS(p,p,2r+1)) \subseteq \APS(m,1,2r+1) \setminus \APS(m,p,2r+1).\]
So assume $P' \in \APS(p,1,2r+1) \setminus \APS(p,p,2r+1)$. If $P = \Psi_{m-p}(P') \in \APS(m,p,2r+1)$ then there exists a witness $w$ for $P$ with $\varepsilon_w \equiv 0 \;(mod \;p)$. As in Proposition 
\ref{prop:odd-reduction1}, let $\xi$ and $\zeta$ be primitive $m^{th}$ and $p^{th}$ roots of unity, respectively. 
Suppose $w$ has pinnacles $P = \{\xi^{a_i}(x_i)\}$ and non-pinnacles $\{ \xi^{b_i}(y_i)\}$, so that $P' = \{ \zeta^{a_i-(m-p)}(x_i)\}$ and thus $b_i \geq a_i \geq m-p$ for all $i$.

In particular, this implies that the preimage $w'  = \psi_{m-p}^{-1}(w)\in \Z_p \wr S_{2r+1}$ of $w$ under $\psi_{m-p}$, is well defined and is a witness for $P'$. However, this means $w'$ satisfies the relation 
\[
\varepsilon_{w'}=\varepsilon_w - (m-p)(2r+1) \equiv 0 \; (mod \; p),
\]
 which contradicts the fact that $P' \notin \APS(p,p,2r+1)$.  
\end{proof}

From Theorem \ref{thm:odd-reduction} the following is now immediate. 

\begin{corollary}\label{cor:reduction}
For any $m,p,r>0$ with $m=pk$ for some $k>0$ we have that 
\[\#\APS(m,p,2r+1) = \#\APS(p,p,2r+1) + \sum_{i=0}^r\binom{2r+1}{i}p^i(k^i-1)(-1)^{i+r}.\]
\end{corollary}

\begin{proof}
Evidently, if $m=p$ the statement is trivial, so suppose $m>p$.
    Since by Lemma \ref{lem:injective} we know that $\Psi_{m-p}$ is injective, then from Theorem \ref{thm:odd-reduction} it follows that 
    \[
    \# \APS(m,p,2r+1) =\#\APS(p,p,2r+1)+\#\APS(m,1,2r+1)-\#\APS(p,1,2r+1),\]
    which by Theorem \ref{thm:pinn-formula} yields the result.
\end{proof}

In \cite[Theorem 4.9]{PinSignedPermutations} it is shown that 
\begin{equation}\label{eq:D-enumeration}
\# \APS(2,2,2r+1) = \#\APS(2,1,2r+1)
+\frac{1}{2}\#\APS(1,1,2r+1)
\end{equation}
and from this derive a closed formula for the number of admissible pinnacle sets for the type $D_n$ signed symmetric group $G(2,2,2r+1)$. Their result relies on a parity condition that holds only when $m=p=2$. While we do not pursue it here, our hope is that by appealing to some higher symmetries in $G(p,p,2r+1)$ the set $\APS(p,p,2r+1)$ also satisfies a relation of the form in Equation \eqref{eq:D-enumeration} and may thus be enumerated using the results herein. 

\bibliography{references}{}
\bibliographystyle{alpha}
\end{document}